\documentclass[a4paper,12pt]{amsart}
\usepackage{amsfonts}
\usepackage{mathrsfs}
\usepackage{amsmath,amssymb,latexsym,amsfonts,amscd}
\usepackage{showlabels}

\newcommand{\bC}{{\mathbb C}}
\newcommand{\bQ}{{\mathbb Q}}
\newcommand{\bP}{{\mathbb P}}
\newcommand{\bZ}{{\mathbb Z}}

\newcommand\Spec{\text{\rm Spec}}
\newcommand\Proj{\text{\rm Proj}}

\newtheorem{thm}{Theorem}[section]

\newtheorem{cor}[thm]{Corollary}
\newtheorem{prop}[thm]{Proposition}

\theoremstyle{definition}
\newtheorem{defn}[thm]{Definition}

\newtheorem{conj}[thm]{Conjecture}

\newtheorem{rem}[thm]{Remark}
\theoremstyle{remark}

\newtheorem{Prbm}{\bf Problem}

\newcommand\cE{{\mathcal{E}}}
\newcommand\cF{{\mathcal{F}}}

\newcommand\cU{{\mathcal{U}}}
\newcommand\cV{{\mathcal{V}}}
\newcommand\cX{{\mathcal{X}}}
\newcommand\cY{{\mathcal{Y}}}
\title{Factoring threefold  divisorial contractions to points}
\author{Jungkai Alfred Chen}
\address{\rm Department of Mathematics, National Taiwan University, Taipei,
106, Taiwan} \email{jkchen@math.ntu.edu.tw}

\address{\rm Research Institute for Mathematical Sciences, Kyoto University, Kyoto 606-8502, Japan}

\thanks{
The  author was partially supported by the NCTS/TPE and National
Science Council of Taiwan. We are indebted to Hacon, Hayakawa,
Kawakita, Kawamata, Koll\'ar and Mori for many useful discussion,
comments and corrections. This work was done during a visit of the
author to the RIMS, Kyoto University. The author would like to thank
the University of Kyoto for its hospitality.}

\begin{document}
\begin{abstract}
We show that terminal $3$-fold divisorial contraction to a point of
index $>1$ with non-minimal discrepancy may be factored into a
sequence of flips, flops and divisorial contractions to a point with
minimal discrepancies.
\end{abstract}

\maketitle
%
\section{introduction}

In minimal model program, the elementary birational maps consists of
flips, flops and divisorial contractions. In dimension three, after
the milestone work of Mori (cf. \cite{Mo82}), these maps are
reasonably well understood while there are many recent progresses in
describing these birational maps explicitly. The geometry of flips
and flops in dimension three can be found in the seminal papers of
 Koll\'ar  and Mori (cf. \cite{Ko89, KM92, Mo88}).
Divisorial contractions to a curve was studies by Cutkosky and
intensively by Tziolas (cf. \cite{Cu, Tz03,Tz05,Tz09}). Divisorial
contractions to points are most well-understood. By results of
Hayakawa, Kawakita, and Kawamata(cf. \cite{HaI,HaII,Kk01,Kk05,Kk11,
Ka}),  it is now known that divisorial contractions to higher index
points in dimension three are weighted blowups (under suitable
embedding) and completely classified. It is expected that all
divisorial contractions to points can be realized as weighted
blowups.

Let $f:Y \to X$ be a divisorial contraction to a point $P \in X$ of
index $n>1$ in dimension three. We say that $f$ has minimal
discrepancy if the discrepancy of $f$ is the minimal possible $1/n$
(cf. w-morphism in \cite{CH}). Divisorial contractions to higher
index points with minimal discrepancies play a very interesting role
for the following two reasons.
\begin{enumerate}
\item For any terminal singularities $P \in X$ of index $n >1$, there exists a
partial resolution $X_n \to \ldots X_0:=X$ such that each $X_n$ has
only terminal Gorenstein singularities and each $X_{i+i} \to X_i$ is
a divisorial contraction to a point with minimal discrepancy (cf.
\cite{HaII}).

\item For any flipping contraction or divisorial contraction to a
curve, by taking a divisorial extraction over the highest index
point with minimal discrepancy, one gets a factorization into
"simpler" birational maps (cf. \cite{CH}).
\end{enumerate}

On the other hand, divisorial contractions to  points with
non-minimal discrepancies  are rather special. For example, if $P
\in X$ is of type cAx/2, cAx/4 or cD/3, then there is no divisorial
contraction with non-minimal discrepancy. The purpose of this note
is to show that divisorial contractions to a higher index point with
non-minimal discrepancies can be factored into divisorial
contractions of minimal discrepancies, flips and flops (cf.
\cite{CH}).


In fact, let $f: Y \to  X$ be a divisorial contraction to a point $P
\in X$ of index $n>1$. Suppose that the discrepancy of $f$ is $a/n >
1/n$. If $Y$ has only Gorenstein singularities, then by the
classification of \cite{Mo82, Cu}, one has that $X$ is Gorenstein
unless $f: Y \to X$ is a divisorial contraction to a quotient
singularity $P \in X$ of type $\frac{1}{2}(1,1,1)$ with discrepancy
$\frac{1}{2}$. Therefore, we may and do assume that $Y$ has some
non-Gorenstein point $Q \in Y$ of index $p$. We thus consider a
divisorial contraction over $Q$ with minimal discrepancy.

\begin{thm}
Let $f: Y \to X$ be an extremal contraction to a point $P \in X$ of
index $n>1$ with exceptional divisor $E$. Let $Q \in Y$ be a point
of highest index $p$ in $E\subset Y$ and  $g: Z \to Y$ be an
extremal extraction with discrepancy $\frac{1}{p}$. Then the
relative canonical divisor $-K_{Z/X}$ is nef.
\end{thm}

Notice that the relative Picard number $\rho(Z/X)=2$. Therefore, we
are able to play the so called 2-ray game. As a consequence, there
is a flip or flop $Z \dashrightarrow Z^+$. By running the minimal
model program of $Z^+/X$, we have $Z \dashrightarrow Z^\sharp
\stackrel{g^\sharp}{\to} Y^\sharp \stackrel{f^\sharp}{\to} X$, where
$Z \dashrightarrow Z^\sharp$ consists of a sequence of flips and
flops, $Z^\sharp \to Y^\sharp$ is a divisorial contraction. Let
$F_{Y^\sharp}$ (resp., $F_{Z^\sharp}, E_{Z^\sharp}$) be the proper
transform of $F$ (resp. $F,E$) in $Y^\sharp$ (resp. $Z^\sharp$).

In fact, we have the following
more precise description.

\begin{thm}
Keep the notation as above. We have that $f^\sharp$ is a divisorial
contraction to $P \in X$ with discrepancy
$\frac{a'}{n}<\frac{a}{n}$. Moreover, $g^\sharp$ is a divisorial
contraction to a singular point $Q' \in F_{Y^\sharp} $ of index $p'$
with discrepancy $\frac{q'}{p'}$. We may rite ${g^\sharp}^*
F_{Y^\sharp}=  F_{Z^\sharp}+ \frac{\frak q}{p'}E_{Z^\sharp}$, then
$$\frac{a}{n}=\frac{a'}{n} \cdot \frac{q'}{p'} +\frac{\frak q}{p'}.$$
More specifically,  exactly one of the following holds.
\begin{enumerate}
\item If $P \in X$ is of type other than cE/2, then $Q'$ is a point of index $n$, and $g^\sharp$ has discrepancy $\frac{a''}{n}$ with
$a'+a''=a$.

\item If $P \in X$ is of type  cE/2, then $Q'$ is a point of index $p'=3$, and $g^\sharp$ has
minimal discrepancy $\frac{1}{3}$.
\end{enumerate}
\end{thm}

As an immediate corollary by induction on discrepancy $a$, we have:

\begin{cor}
For any divisorial contraction $Y \to X$ to a point $P \in X$ of index $n>1$ with discrepancy $\frac{a}{n} >\frac{1}{n}$. There exsits a sequence of birational maps $$Y=:X_n \dashrightarrow \ldots \dashrightarrow X_0=:X$$ such that each map $X_{i+1} \dashrightarrow X_{i}$ is one of the following:\\

\begin{enumerate}
\item a divisorial extraction over a point of index $r_i >1$ with minimal discrepancy
$\frac{1}{r_i}$;

\item a divisorial contraction to a point of index $r_i >1$ with minimal discrepancy
$\frac{1}{r_i}$;

\item a flip or flop.

\end{enumerate}

\end{cor}

We now  briefly explain the idea. According the 2-ray game, we have
the following diagram of birational maps.
$$\begin{CD}
Z @>{\dashrightarrow}>> Z^\sharp \\
@VgVV @VV{g^\sharp}V\\
Y @>>> Y^\sharp \\
@VfVV @VV{f^\sharp}V \\
X@>{=}>> X
\end{CD}
$$
Notice that in this diagram the order of exceptional divisors of the
tower $Z^\sharp \to Y^\sharp \to X$ and $Z \to Y \to X$ are
reversed.  The usual difficulty to understand the diagram explicitly
is that we need to determine the center of $E_{Z^\sharp}$ in
$Y^\sharp$.

On the other hand, since $f: Y \to X$ is a weighted blowup, one can
embed $X$ into a toric variety $\cX_0$ and understand $f:Y \to X$ as
the proper transform of a toric weight blowup $\cX_1 \to \cX_0$,
which is nothing but a subdivision of a cone along a vector $v_1$.
If $Z \to Y$ can be realized as the proper transform of a toric
weighted blowup $\cX_2 \to \cX_1$ over the origin of the standard
coordinate charts, then we can view $\cX_2 \to \cX_1$ as a toric
weighted blowup along a vector $v_2$. Therefore, the tower $\cX_2
\to \cX_1 \to \cX_0$ is obtained by subdivision along vectors $v_1$
and then $v_2$.

We may reverse the ordering of $v_1,v_2$ (under mild combinatorial
condition) by considering a tower $\cX'_2 \to \cX'_1 \to \cX_0$ of
toric weighted blowup by subdivision along $v_2$ and then $v_1$. The
proper transforms of $X$ in this tower then gives $Z' \to Y' \to X$.
Clearly, this is  a tower reversing the order of exceptional
divisors of $Z \to Y \to X$ by construction.  Notice that the proper
transform $Y'$ and $Z'$ is not necessarily terminal, a priori.

In section 2, we recall and generalize the construction of weighted
blowup. We also derive a criterion for $-K_{Z/X}$ being nef.
Moreover, we show that if the tower $Z \to Y \to X$ can be embedded
into a tower of weighted blowup $\cX_2 \to \cX_1 \to \cX_0$ and
$-K_{Z/X}$ is nef, then the output of 2-ray game coincides with the
output by "reversing order of vectors" of the tower of weighted
blowups.

In Section 3, we study divisorial contractions with non-minimal
discrepancies case by case. We see that the divisorial extraction $Z
\to Y$ over a point of index $>1$ usually give a tower $Z \to Y \to
X$ such that $-K_{Z/X}$ is nef and it can be embedded into a tower
of weighted blowups $\cX_2 \to \cX_1 \to \cX_0$ with vectors
$v_1,v_2$. Indeed this is always the case if $Z \to Y$ is a
contraction over a point of highest index.  The theorems then
follows easily.

We always work over complex number field $\bC$ and in dimension
three. We assume that threefold $X,Y$ are $\bQ$-factorial. We freely
use the standard notions in minimal model program such as terminal
singularities, divisorial contractions, flips, and flops. For the
precise definition, we refer to \cite{KM98}.
\section{Preliminaries}

\subsection{weighted blowups}
We recall the construction of weighted blowups by using the toric
language.

Let $N=\bZ^d$ be a free abelian group of rank $d$ with standard
basis $\{e_1,...,e_d\}$. Let $v=\frac{1}{n}(a_1,...,a_d) \in \bQ^d$
be a vector. We may assume that $gcd(n,a_1,...,a_d)=1$. We consider
$\overline{N}:=N+\bZ v$. Clearly, $N \subset \overline{N}$. Let $M$
(resp. $\overline{M}$) be the dual lattice of $N$ (resp.
$\overline{N}$).

Let $\sigma$ be the cone of first quadrant, i.e. the cone generated
by the standard basis $e_1,...,e_d$ and $\Sigma$ be the fan consists
of $\sigma$ and all the subcones of $\sigma$. We have
$$\begin{array}{l}
\cX_{N,\Sigma}:=\Spec \bC[\sigma^\vee \cap M]= \bC^d,\\
\cX_{\overline{N},\Sigma}:=\Spec \bC[\sigma^\vee \cap \overline{M}]=
\bC^d/\frac{1}{n}(a_1,...,a_d),
\end{array}
$$

Let $v_1=\frac{1}{r_1}(b_1,...,b_d)$ be a primitive vector in
$\overline{N}$. We assume that $b_i \in \bZ_{>0}$ and
$gcd(r_1,b_1,...,b_d)=1$. We are interested in the weighted blowup
of $ o \in
\bC^d/\frac{1}{n}(b_1,...,b_d)=\cX_{\overline{N},\Sigma}=:\cX_0$
with weights $v_1=\frac{1}{r_1}(b_1,...,b_d)$ which we describe now.

Let $\overline{\Sigma}$ be the fan obtained by subdivision of
$\Sigma$ along $v_1$. One thus have a toric variety
$\cX_{\overline{N},\overline{\Sigma}}$ together with the natural map
$ \cX_{\overline{N},\overline{\Sigma}} \to
\cX_{\overline{N},\Sigma}$. More concretely, let $\sigma_i$ be the
cone generated by $\{e_1,...,e_{i-1},v_1,e_{i+1},...,e_d\}$, then
$$\cX_1:=\cX_{\overline{N},\overline{\Sigma}}= \cup_{i=1}^d \cU_i,$$
where $\cU_i=\cX_{\overline{N},\sigma_i}=\Spec \bC[ \sigma_i^\vee
\cap \overline{M}]$. We always denote the origin of $\cU_i$ as
$Q_i$.

\subsection{tower of toric weighted blowups}

Let us look at $\cU_i$, which is $\cX_{\overline{N},\sigma_i}$.
Suppose that there is a primitive vector $v_2=\sum \frac{c_i}{r_2}
e_i \in \overline{N}$ such that $v_2$ is in the interior of
$\sigma_i$.


We can write
$$\begin{array}{ll}
v_2&= \frac{1}{n}(c_1e_1+...+c_de_d) \\
&=\frac{1}{p}(q_1 e_1+q_2 e_2+. +q_i v_1+..+q_d e_d),
\end{array}
 $$
for some $q_i \in \bZ_{>0}$.

We denote $w_2=\frac{1}{p}(q_1,...,q_d)$ to be the weight of $v_2$
is the cone $\sigma_i$, or simply the weight of $v_2$ if no
confusion is likely. It is convenient to introduce
$\widehat{w_2}:=\frac{1}{p}(q_1,...,q_{i-1},0,q_{i+1},...,q_d)$.
Then we have
$$ v_2=\frac{q_i}{p}v_1+\widehat{w_2}. \eqno{\sharp}$$

 \noindent {\bf Observation.}\label{toric_sing} Keep the
notation as above. Notice that if
$$\overline{N} \text{ is generated by } \{v_1,v_2,e_1,...,\hat{e_i},...,e_d\}, \eqno{\dagger}$$ then $\cX_{\overline{N},\sigma_i} \cong \bC^d/w_2=
\bC^d/\frac{1}{p}(q_1,q_2,...,q_d)$ has only quotient singularity at
$Q_i$.


We can consider the second weighted blowup with vector $v_2$.
Let
$\overline{\overline{\Sigma}}$ be the fan obtained by subdivision of
$\sigma_i$ along $v_2$. One thus have a toric variety
$\cX_{\overline{N},\overline{\overline{\Sigma}}}$. Similarly, let $\tau_j$ be the cone generated by
$$\left\{ \begin{array}{ll}
\{e_1,...,e_{j-1},v_2,e_{j+1},...,e_{i-1},v_1,e_{i+1},...,e_d\},&\text{ if } j \ne i \\
\{e_1,...,e_{i-1},v_2,e_{i+1},..,e_d\},&\text{ if } j = i \\
\end{array} \right.$$

Then $$\cX_2:=\cX_{\overline{N},\overline{\overline{\Sigma}}}= (\cup_{j=1}^d \cV_j) \bigcup (\cup_{k \ne i} \cU_k),$$
where $\cV_j=\Spec \bC[ \tau_j^\vee \cap \overline{M}]$.
Let $w_2=\frac{1}{p}(q_1,...,q_d)$. Then the weighted blowup $\cup_{j=1}^d \cV_j \to \cU_i$ with vector $v_2$ can be conseidered as a weighted blowup with weights $w_2$.

\begin{defn}We say that $\cX_1 \to \cX_0$ is the weighted
blowup with with vector $v_1$ or we say that $\cX_1 \to \cX_0$ is
the weighted blowup with weights $w_1=\frac{1}{r_1}(c_1,...,c_d)$.
Similarly, we say that $\cX_2 \to \cX_1$ is the weighted blowup with
vector $v_2$ or with weights $w_2=\frac{1}{p}(q_1,...,q_d)$.
\end{defn}

Notice that by construction $v_2= \frac{1}{r_2}(c_1e_1+...+c_de_d)$
with $c_i>0$ for all $i$. We can consider $\cX'_1 \to \cX_0$ the
weighted blowup  with vector $v_2$, then we have that $\cX'=\cup
\cU'_i= \cup \Spec \bC [{\sigma'_1}^\vee \cap \overline{M}]$ with
$\sigma'_i$ the cone generated by
$\{e_1,...,e_{i-1},v_2,e_{i+1},...,e_d\}$. Then clearly,
$$\cU'_i=\Spec \bC [{\sigma'_i}^\vee \cap \overline{M}]=\Spec \bC [\tau_i^\vee \cap \overline{M}] = \cV_i.$$
Notice also that the exceptional divisor
$\mathcal{F}$ of $\cX_2 \to \cX_1$ and the exceptional divisor
$\mathcal{F}'$ of $\cX'_1 \to \cX_0$ defines the same valuation given by the cone generated by $v_2$.

Suppose furthermore that $v_1$ is in the interior of $\sigma'_k$ for
some $k$. Then we can consider a weighted blowup $\cX'_2 \to \cX'_1$
with vector $v_1$. Notice also that the exceptional divisor
$\mathcal{E}$ of $\cX_2 \to \cX_1$ and the exceptional divisor
$\mathcal{E}'$ of $\cX'_1 \to \cX_0$ defines the same valuation
given by the cone generated by $v_1$.

\begin{rem} \label{interchange}
We say that $v_1$ and $v_2$ are {\it interchangeable} if $v_2$ is in
the interior of $\sigma_i$ for some $i$ and $v_1$ is in the interior
of $\sigma'_k$ for some $k$. It is easy to see that $v_1,v_2$ are
interchangeable if $b_j c_l \ne b_l c_j$ for all $j \ne l$.
\end{rem}

 In this situation, we say that the tower of weighted
blowups $\cX'_2 \to \cX'_1 \to \cX_0$ (with vectors $v_1,v_2$
successively) is obtained by reversing the order of the tower of
weighted blowup $\cX_2 \to \cX_1 \to \cX_0$ with vectors $v_2,v_1$.
 We have the following diagram
$$ \begin{CD} \cX_2 @>{\dashrightarrow}>> \cX'_2\\
@V{v_2}VV @VV{v_1}V \\
\cX_1 @. \cX'_1 \\
@V{v_1}VV @VV{v_2}V \\
\cX_0 @>{=}>> \cX_0.
\end{CD}$$

\subsection{complete intersections}
The toric variety $\cX_0 \cong \bC^d/v$ is a quotient by
$\bZ_n$-action with weights $\frac{1}{n}(a_1,...,a_d)$. For any
semi-invariant $\varphi= \sum \alpha_{i_1,...,i_d}
x_1^{i_1}...x_d^{i_d}$, we define
$$wt_v(\varphi):=\min \{ \sum_{j=1}^d \frac{ a_j }{n} i_j | \alpha_{i_1,...,i_d} \ne 0 \}.$$
For any vector $v' \in \overline{N}$, we define $wt_{v'}$ similarly.

Given a cyclic quotient of complete intersection variety, i.e. an
embedding $X=(\varphi_1=\varphi_2=...=\varphi_k=0) \subset
\bC^d/v=\cX_0$, where each $\varphi_i$ is a semi-invariant. Let
$\cX_1 \to \cX_0$ be a weighted blowup with vector $v_1$ and
exceptional divisor $\cE$. Let $Y$ be the proper transform of $X$ in
$\cX_1$. Then we say that the induce map $\phi: Y \to X$ is the
weighted blowup with vector $v_1$. Note that its exceptional set is
$E=\cE \cap Y$.

Quite often, we need to embed $X$ into a another ambient space. For
example, write $\varphi_k= f_0+f_1 f_2$ with $f_1$ being a
semi-invariant. We set
$v':=v+wt_v(f_1)e_{d+1}=(\frac{a_1}{n},\ldots, \frac{a_d}{n},
wt_v(f_1))$.

We then consider $X' \subset \cX'_0:=\bC^{d+1}/v'$ by setting:
$$\left\{ \begin{array}{ll}
\varphi'_j:=\varphi_j, &\text{for } 1 \le j \le k-1,\\
\varphi'_k:=f_0+x_{d+1}f_2, &\\
\varphi'_{k+1}:=x_{d+1}-f_1.&
\end{array} \right.
$$ It is easy to see that $X \cong X'$.

Let $\cX_1 \to \bC^d/v:=\cX_0$ be a weighted blowup with weights
$v_1=\frac{1}{r_1}(c_1,...,c_d)$. We set
$v'_1=(\frac{c_1}{r_1},...,\frac{c_d}{r_1}, wt_{v_1}(f_1))$ and let
$\cX'_1 \to \cX'_0$ be the weighted blowup with weights $v'_1$. Let
$Y, Y'$ be their proper transform in $\cX_1, \cX'_1$ respectively.
Then it is straightforward to check that $Y \cong Y'$ canonically.
Indeed, the isomorphism follows from the canonical isomorphism of $Y
\cap \cU_j \cong Y' \cap \cU'_j$ for $j \le d$ and $Y' \cap
\cU'_{d+1}=\emptyset$.

\begin{defn}\label{comp_reembed}
The weighted blowups $Y \to X$ with weights $v_1$ and $Y' \to X'$
with weight $v'_1$ are said to be compatible if the equations and
weights are defined as above.
\end{defn}

\subsection{2-ray game}
Turning back to the study of terminal threefolds.  We may write $P
\in X$ as $(\varphi=0) \subset \bC^4/ v_0$ (resp.
 $\varphi_1=\varphi_2=0 \subset \bC^5/ v_0$). By a weighted blowup
$f: Y \to X$ with weights $v_1=\frac{1}{r_1}(b_1,...,b_4)$ (resp.
$v_1=\frac{1}{r_1}(b_1,...,b_5)$), we denote the standard coordinate
chart as $U_i=\cU_i \cap Y$, $i=1,\ldots, 4$ (resp. $i=,\ldots, 5$)
and let $Q_i$ be the origin of $\cU_i$.

Given a divisor $D$ on any of the birational model, adding a
subscript, e.g. $D_X, D_Y$, will denote its proper transform in $X,
Y$ respectively (if its center is a divisor). Similarly for a
$1$-cycle $l$.

Let us consider a divisorial contraction $f: Y \to X$ to a point of
index $r$ with discrepancy $\frac{a}{n} > \frac{1}{n}$. Let $E$ be
the exceptional divisor of $f$.

Suppose that $Q_i \subset
Y$ is point of index $p>1$. We consider $g: Z \to Y$ be a divisorial contraction
 with discrepancy $\frac{1}{p}$. Let $F$ be the exceptional
divisor of $g$. We may write $g^*E=E_Z+\frac{\frak q}{p}F$.


Let $D_0 \ne E$ be a divisor on $Y$ passing through $Q_i$ such that
$l_0:=D_0 \cdot E$ is irreducible (possibly non-reduced). Let
$D_{0,X},D_{0,Z}$ be its proper transform on $X,Z$ respectively.
Notice that we have $$f^*D_{0,X}=D_0+ \frac{c_0}{n} E, \quad g^*D_0
= D_{0,Z}+\frac{q_0}{p} F$$ for some $c_0,q_0 \in \bZ_{>0}$.

Notice also that $l_{0,Z}=D_{0,Z} \cdot {E_Z}$. Clearly, we have
$$ g^* l_0=l_{0,Z}+ \frac{q_0}{p} l_F,$$
as a $1$-cycle, where $l_{0,Z}$ is the proper transform and $l_F:= F
\cdot E_Z$.


It is easy to see that
$$l_0 \cdot K_Y= D_0 \cdot E  \cdot K_Y = \frac{-ac_0}{n^2} E^3<0.$$
We also have
$$l_{0,Z} \cdot K_Z= D_{0,Z} \cdot E_Z \cdot K_Z=D_j\cdot E \cdot K_Y
+ \frac{\frak q q_0}{p^3}F^3=\frac{-ac_0}{n^2} E^3+\frac{\frak q
q_0}{p^3}F^3. \eqno(1)$$

Now for any curve $l \subset E$. Since $\rho(Y/X)=1$, we have that
$l$ is proportional to $l_0$ as a $1$-cycle. In other words, for any
divisor $D$ on $Y$,
$$l \cdot D = \alpha l_0 \cdot D,$$ for some $\alpha$.
We set $c=\alpha c_0$ (not necessarily an integer).
 Therefore,
$ l \cdot K_Y =  \frac{-ac}{n^2}E^3$.

We can write $g^*l=l_Z+\frac{q}{p}l_F$ for $\rho(Z/X)=2$ and the cone
of curves clearly generated by $l_Z$ and $l_F$ (note that we did not
assume that $q$ is an integer here). Similar computation shows that
$$l_Z \cdot K_Z= l \cdot K_Y +
\frac{q_iq}{p^3}F^3=\frac{-ac}{n^2}E^3+\frac{ \frak q q}{p^3}F^3.
\eqno(2)$$

Notice that $$l \cdot_E l_0 = l \cdot_Y D_0=\frac{c}{c_0} l_0 \cdot
D_0 = \frac{c}{c_0} D_0^2 E = \frac{c}{c_0} \frac{c_0^2}{r^2}E^3 =
\frac{c c_0}{r^2} E^3.$$

Also this quantity can be computed by $$ g^*l \cdot_{E_Z} g^*
l_0=g^*l \cdot_Z g^* D_0=l_Z \cdot D_{0,Z}+ \frac{q \frak q
q_0}{p^3}F^3 =l_Z \cdot_{E_Z} l_{0,Z} +\frac{q \frak q
q_0}{p^3}F^3.$$

If $l \ne l_0$, then $l_Z \cdot_{E_Z} l_{0,Z}  \ge 0$. So we have
$$\frac{c c_0}{n^2} E^3 \ge \frac{q \frak q
q_0}{p^3}F^3. $$ Compare with $(2)$, we have that for $l \ne l_0$,
$$q_0 l_Z \cdot K_Z \le  \frac{c}{n^2}(c_0-aq_0) E^3. \eqno(3)$$

We thus conclude the following criterion.
\begin{prop}
Let $D_0 \ne E$ be a divisor on $Y$ passing through a point $Q_i$ of
index $p$ such that $l_0:=D_0 \cdot E$ is irreducible (possibly
non-reduced). Let $g: Z \to Y$ be a extremal contraction to $Q_i$.
Let $D_{0,X},D_{0,Z}$ be the  proper transform of $D_0$ on $X,Z$
respectively. We write $$f^*D_{0,X}=D_0+ \frac{c_0}{n} E, \quad
g^*D_0 = D_{0,Z}+\frac{q_0}{p} F, \quad g^*E=E_Z+\frac{\frak
q}{p}F$$ for some $c_0,q_0, \frak q \in \bZ_{>0}$. Then $-K_{Z/X}$
is nef if the following inequalities holds:
$$ \left\{
\begin{array}{l}
T(f,g,D_0):=\frac{-ac_0}{n^2}
E^3+\frac{\frak q q_0}{p^3}F^3 \le 0,\\
c_0-a q_0 \le 0.
\end{array}
\right.$$
\end{prop}

Indeed, one has a more effective way of calculation by using the "general elephant", if its restriction is irreducible. Let $\Theta \in |-K_Y|$ be an  elephant and $\theta=\Theta|_E$. We have
$$ \begin{array}{l}
g^*\Theta=\Theta_Z+\frac{1}{p}F, \\
f^*\Theta_X=\Theta+\frac{a}{n}E, g^*E=E_Z+\frac{\frak q}{p} F.
\end{array}
$$

Suppose that $\theta$ is irreducible, then one has that $-K_{Z/X}$ is nef
if
$$T(f,g):=\theta_Z \cdot K_Z =\frac{-a^2}{n^2}E^3+\frac{\frak q}{p^3}F^3 \le 0$$
since the second inequality holds automatically.



Suppose now that $-K_{Z/X}$ is nef, then we can play the so-called
"2-ray game" as in \cite{CH}. We have $Z \dashrightarrow Z^\sharp \to
Y^\sharp \to X$, where $Z \dashrightarrow Z^\sharp$ consists of a sequence of
flips and flops, $g^\sharp: Z^\sharp \to Y^\sharp$ is a divisorial contraction.

\begin{prop} Keep the notation as above. We have
$g^\sharp$ contracts $E_{Z^\sharp}$ and $f^\sharp$ is a divisorial contraction to $P \in X$
contracting
$F_{Y^\sharp}$.
\end{prop}

\begin{proof}
Since there are only two exceptional divisors $E_{Z^\sharp}$ and $F_{Z^\sharp}$
on $Z^\sharp$ over $X$. Suppose on the contrary that $g^\sharp$ contracts
$F_{Z^\sharp}$. Then $E_{Y^\sharp}$ is the only exceptional on $Y^\sharp/X$. Moreover,
$\rho(Y^\sharp/X)=1$. We thus have $Y^\sharp \cong Y$ for $E$ and $E_{Y^\sharp}$
clearly defines the same valuation. Then one sees that $Z/Y$ has
exceptional divisor $F$ and $Z^\sharp/Y$ has exceptional divisor $F_{Z^\sharp}$
which again defines the same valuation. Hence $Z \cong Z^\sharp$, which is
absurd.

Notice that $\rho(Y^\sharp/X)=1$, $Y^\sharp$ is terminal $\bQ$-factorial and
$F_{Y^\sharp}$ is the support of the exceptional set. It suffices to show
that $K_{Y^\sharp/X}$ is $-f^\sharp$-ample. Let $\gamma \subset F_{Y^\sharp}$ be a
curve. Pick any very ample divisor $H$ on $Y^\sharp$, then we have
${f^\sharp}^*H_X=H+\mu F_{Y^\sharp}$ for some $\mu>0$. Intersect with $\gamma$, we
have
$$0= \gamma \cdot {f^\sharp}^*H_X=\gamma \cdot H + \mu \gamma \cdot
F_{Y^\sharp}.$$ Hence $\gamma \cdot F_{Y^\sharp} <0$. Now
$$ \gamma \cdot K_{Y^\sharp}= \gamma \cdot a(F_{Y^\sharp},X) F_{Y^\sharp} = \gamma
\cdot a(F,X) F_{Y^\sharp} <0,$$ for the discrepancy of $F$ over $X$ is
positive and depends only on the its valuation.
\end{proof}

\subsection{weighted blowups and 2-ray game}
We fix an embedding $P \in X
\hookrightarrow \cX_0$ such that the divisorial contraction $f:Y \to
X$ is given by the weighted blowup $\cX_1 \to \cX_0$ with weights
$v_1$. That is, $Y$ is the proper transform of $X$ in $\cX_1$.
Let $g:Z \to Y$ be a divisorial contraction with minimal discrepancy over a point $Q_i$ of
index $p >1$

Suppose  that, under such embedding, the following hypotheses holds.
\noindent {\bf Hypothesis $\flat$.}
\begin{enumerate}
\item The divisorial
extraction $g:Z \to Y$ is given by a weighted blowup $\cX_2 \to \cX_1$ over a
point $Q_i$ with vector $v_2$.

\item The vectors $v_1,v_2$ are interchangeable (cf. Remark
\ref{interchange}).

\item $-K_{Z/X}$ is nef.
\end{enumerate}

 Then we have the following diagram.
$$ \begin{CD}
Z^\sharp @<{\dashleftarrow}<< Z @>{\hookrightarrow}>> \cX_2 @>{\dashrightarrow}>> \cX'_2 @<{\hookleftarrow}<< Z'\\
@V{g^\sharp}VV @V{g}VV @V{v_2}VV @VV{v_1}V @VV{g'}V\\
Y^\sharp @.  Y @>{\hookrightarrow}>> \cX_1 @.                   \cX'_1 @<{\hookleftarrow}<< Y'\\
@V{f^\sharp}VV @V{f}VV @V{v_1}VV @V{\eta}V{v_2}V @VV{f'}V\\
X @<{=}<< X @>{\hookrightarrow}>> \cX_0 @>{=}>>             \cX_0 @<{\hookleftarrow}<< X,\\
\end{CD}$$
where $Z^\sharp \to Y^\sharp \to X$ is the output of the two-rays game and $Z',
Y'$ are proper transform of $X$ in $\cX'_2,\cX'_1$ respectively.

\begin{thm} \label{compatible}
Keep the notation as above and suppose that Hypothesis $\flat$
holds. Then $Y^\sharp \cong Y'$ and $Z^\sharp \cong Z'$. In
particular, both $f^\sharp$ and $g^\sharp$ are weighted blowups and
both $f'$ and $g'$ are divisorial contractions to a point.
\end{thm}

\begin{proof}
Let $\cF$ be the exceptional divisor of $\cX_2 \to \cX_1$. It is the
excetional divisor induced by the vector $v_2$. Hence its proper
transform $\cF'$ in $\cX'_1$ is the exceptional divisor of
$\eta:\cX'_1 \to \cX_0$. Recall that by the construction in
Subsection 2.2, there is a canonical isomorphism $\cV_i \cong
\cU'_i$ for some $i$, where $\cV_i \subset \cX_2$ and $\cU'_i
\subset \cX'_1$ are coordinate charts. Surely, we have an induced
isomorphism $Z \cap \cV_i \cong Y' \cap \cU'$. Since $F$ is
irreducible and
$$F \cap \cV_i = (\cF \cdot Z)  \cap \cV_i \cong (\cF' \cdot Y') \cap \cU'_i.$$
It follows that $F_{Y'}:=\cF' \cdot Y'$ is irreducible,
which coincides with the exceptional set. On the other hand, the
proper transform of $F$ in $Y^\sharp$ is $F_{Y^\sharp}$, which is
the exceptional divisor of $f^\sharp$. One sees immediately that
$F_{Y'}$ and $F_{Y^\sharp}$ define the same valuation in the
function field.

Note that $-\cF'$ is clearly $\eta$-ample. It follows that $-F_{Y'}$ is $f'$-ample. Hene
we  have
$$ Y^\sharp= \Proj ( \oplus_{m \ge 0} f^\sharp_* \mathcal{O}(-mF_{Y^\sharp})) \cong \Proj ( \oplus_{m \ge 0} f'_* \mathcal{O}(-mF_{Y'})) = Y'.$$

The proof for $Z^\sharp \cong Z'$ is similar.
\end{proof}


\section{case studies}
In this section we study divisorial contractions to a higher index
point with non-minimal discrepancy case by case. For each case, we
consider the extraction over a higher index point. We shall show
that the Hypothesis $\flat$ holds for all tower by extracting over a
highest index point and for some other extraction over another
higher index point. Hence, in particular,  Theorem 1.1 follows.

Moreover the output of 2-ray game and interchanging vectors of
weighted blowups coincide. Hence we end up with a diagram for each
case, where every vertical map is a weighted blowup. Theorem 1.2
then follows by checking the diagram for each case.

\subsection{discrpancy=$4/2$ over a cD/2 point}
Let $Y \to X$ be a divisorial contraction to a cD/2 point $P \in X$
with discrepancy $2$. By Kawakita's work (cf. \cite{Kk11}), it is known that there
exists an embedding
$$ \left\{ \begin{array}{l} \varphi_1: x_1^2+x_4x_5+p(x_2,x_3,x_4)=0, \\
\varphi_2: x_2^2+q(x_1,x_3,x_4)+x_5=0 \end{array} \right.$$ with
$v=\frac{1}{2}(1,1,1,0,0)$. Also $f$ is the weighted blowup with
weights $v_1=(4l+1,4l,2,1,8l+1)$ or $(4l,4l-1,2,1,8l-1)$.

We treat this case in greater detail. The remaining cases can be
treated similarly. Note that we can write
$p(x_2,x_3,x_4)=x_4p_1(x_2,x_3,x_4)+p_0(x_2,x_3)$. Therefore,
replacing $x_5$ by $x_5+p_1(x_2,x_3,x_4)$, we may assume that
$\varphi_1=x_1^2+x_4x_5+p(x_2,x_3)$.

\noindent {\bf Case 1.} $v_1=(4l+1,4l,2,1,8l+1)$.\\
Note that $wt_{v_1}(p(x_2,x_3)) \ge 8l+1$, $wt_{v_1}(q(x_1,x_3,x_4))
\ge 8l$.

\noindent {\bf Step 1.} We search for points in $Y$ with index $>1$. This can only happen over $Q_i$. Clearly, $Q_1, Q_2 \not \in Y$.

We first look at $Q_3$. By computation of local charts, one sees that $Q_3 \in
\cX_1$ is a quotient singularity of type $\frac{1}{4}(1,2,1,3,3)$.


\noindent {\bf Claim 1.} $Q_3 \not \in Y$ and $x_3^{4l} \in \varphi_2$. \\
To see this, according to Kawakita's description, there is only one
non-hidden non-Gorenstein singularity and also the hidden
singularities has index at most $2$. Hence $Q_3 \not \in Y$. In
other words, one must have either $x_3^{4l+1} \in \varphi_1$ or
$x_3^{4l} \in \varphi_2$. Note that $x_3^{4l+1} \not \in \varphi_1$
otherwise $\varphi_1$ is not a semi-invariant. We thus conclude that
$x_3^{4l} \in \varphi_2$.

We can see  that $Q_5 \in \cX_1$ is a quotient singularity  of type
$\frac{1}{2(8l+1)}(6l+1,10l+1,1,12l+2,4l)$ with index $2(8l+1)$. We
set $w_2=\frac{1}{2(8l+1)}(6l+1,10l+1,1,12l+2,4l)$ so that
$v_2=\frac{1}{2}(2l+1,2l+1,1,2,4l)$.


\noindent {\bf Remark.} The point $Q_4 \in Y$ is a "hidden" cD/2
point (see \cite[p.68]{Kk05}). By the classification of Hayakawa
(cf. \cite{HaII}), any divisorial contraction $g: Z \to Y$ has the
property that $g^*E=E_Z+\frac{t}{2}F$ with $t>0$ even. Therefore,
$a(F,X)=\frac{t}{2}2+\frac{1}{2} >2$. Hence our theorem does not
hold for arbitrary extraction over a point $Q$ of index $r>1$.

\noindent {\bf Step 2.} The weighted blowup $\cX_2 \to \cX_1$ with weights
$w_2$ gives a divisorial contraction $g: Z \to Y$ of discrepancy
$\frac{1}{2(8l+1)}$.\\
To see this, note that the local equation of $Q_5$ is given by
$$ \left\{ \begin{array}{l} \overline{\varphi_1}: \overline{x_1}^2+\overline{x_4}+\overline{p}(\overline{x_2},\overline{x_3})=0, \\
\overline{\varphi_2}: \overline{x_2}^2+\overline{q}(\overline{x_1},\overline{x_3},\overline{x_4})+\overline{x_5}=0.
\end{array} \right.$$
We have natural isomorphism between $o \in
\bC^3/\frac{1}{2(8l+1)}(6l+1,10l+1,1)=:\cY_1$ and $Q_5 \in \bC^5/w_2$. The
only extremal extraction over $o$ with discrepancy
$\frac{1}{2(8l+1)}$ is the Kawamata blowup $\cY_2 \to \cY_1$, which is the weighted blowup with weights $\overline{w_2}=\frac{1}{2(8l+1)}(6l+1,10l+1,1)$. Since
$\overline{x_3}^{4l} \in \overline{\varphi_2}$, one sees that
$$\left\{
\begin{array}{l}wt_{w_2}(\overline{x_4})= wt_{\overline{w_2}}( \overline{x_1}^2)=wt_{\overline{w_2}}( \overline{x_1}^2+\overline{p}(\overline{x_2},\overline{x_3})), \\
wt_{w_2} (\overline{x_5} )= wt_{\overline{w_2}}(
\overline{x_3}^{4l})= wt_{\overline{w_2}} (
\overline{x_2}^2+\overline{q}(\overline{x_1},\overline{x_3},\overline{x_4})).
\end{array} \right.$$

Therefore, the weighted blowup $\cX_2 \to \cX_1$ with weights $w_2$
and $\cY_2 \to \cY_1$ are compatible (cf. Subsection 2.3). In
particular, the only divisorial contraction $g: Z \to Y$ of
discrepancy $\frac{1}{2(8l+1)}$ is obtained by weighted blowup with
weights $w_2$ (with vector $v_2$). This verifies Hypothesis
$\flat$(1). The hypothesis $\flat$(2) can be verified trivially.

\noindent{\bf Step 3.} We now checked the numerical conditions for
2-ray game. By Kawakita's Table (cf. \cite[Table 1,2,3]{Kk05}), we
have
$$E^3= \frac{2}{2(8l+1)}, \quad F^3=\frac{(2(8l+1))^2}{(6l+1)(10l+1)}.$$

Note that the exceptional divisor $E$ can be realized as a
$\bZ_2$-quotient of complete intersection
$$\tilde{E}:=(\varphi_{1,8l+2}=\varphi_{2,8l}=0) \subset \bP(4l+1,4l,2,1,8l+1),$$
where $\varphi_{i,k}$ denotes the homogeneous part of $\varphi_i$ of
$v_1$-weight $k/2$. Indeed, if we pick $D_{0,X}=(x_3=0)$, which is
an elephant in $|-K_X|$,  we have that $E \cap D_0$ is defined by
$\bZ_2$-quotient of the complete intersection
$$\left\{
\begin{array}{l}
x_3=0,\\
{\varphi_{1,8l+2}}_{|x_3=0}=x_1^2+x_4x_5,\\
{\varphi_{2,8l}}_{|x_3=0}=x_2^2+ q_{8l}(x_1,0,x_4).$$
\end{array}
\right.$$

If $q_{8l}(x_1,0,x_4)$ is not a perfect square, then this is clearly
irreducible. If $q_{8l}(x_1,0,x_4)$ is a perfect square, then this
is reducible on $\tilde{E}$ but irreducible on $E$ after the
$\bZ_2$-quotient.

 Therefore, we can simply check
$$T(f,g)=\frac{1}{2(8l+1)}(-8 + \frac{4l}{(6l+1)(10l+1)}) <0$$
to conclude that $-K_Z/X$ is nef. This verifies Hypothesis
$\flat$(3).


\noindent {\bf Step 4.} The weighted blowup $\cX' \to \cX_0$ with vector
$v_2$ gives a divisorial contraction $f': Y' \to X$ of discrepancy
$\frac{1}{2}$.\\
This follows from Theorem \ref{compatible}.
In fact, we can check this directly as well by considering a
re-embedding $\overline{X} \subset \bC^4/\frac{1}{2}(1,1,1,0)$
defined by
$$\varphi: x_1^2+x_2^2x_4+q(x_1,x_3,x_4) x_4+p(x_2,x_3,x_4)$$ with
$x_3^{4l}x_4 \in \varphi$. Let
$\overline{v_2}=\frac{1}{2}(2l+1,2l+1,1,2)$, then one sees that the
weighted blowup $Y' \to X$ with weight $v_2$ is compatible with
weighted blowup of $\overline{Y} \to \overline{X}$ with weigh
$\overline{v_2}$. It is easy to see that $wt_{v_1}(p) \ge 8l+1$
implies that $wt_{\overline{v_2}}(p)
> 2l$ and $wt_{v_1}(q) \ge 8l$ implies that $wt_{\overline{v_2}}(q
x_4)\ge 2l+1$. Therefore, the weighted blowup  $\overline{Y} \to
\overline{X}$ with weight $\overline{v_2}$ is indeed the weighted
blowup given in Proposition 5.8 of \cite{HaII}, which is a
divisorial contraction with minimal discrepancy $\frac{1}{2}$. Hence
so is $Y' \to X$.


\noindent{ \bf Step 5.} One sees that $v_1=\frac{6l+1}{2}
e_1+\frac{6l-1}{2}e_2+\frac{3}{2}e_3+v_2+\frac{12l+2}{2}e_5$.
Therefore, one consider the weighted blowup $\cX'_2 \to \cX'_1$ with
weights $w'_2=\frac{1}{2}(6l+1,6l-1,3,2,12l+2)$ over $Q'_4 \in
\cX'_1$. Let $Z'$ be the proper transform in $\cX'_2$.  Notice that
$Z' \to Y'$ is a divisorial contraction over $Q'_4$ with discrepancy
$\frac{3}{2}$. This is indeed the map in Case 1 of Subsection 3.2
(after re-embedding into $\bC^4/v$ as in Step 4.)

We summarize this case into following diagram.
$$\begin{CD}
Z @>{\dashrightarrow}>> Z' \\
@V{\frac{1}{2(8l+1)}}V{wt=w_2}V @V{\frac{3}{2}}V{wt=w'_2}V \\
Q_5 \in Y @.  Y' \ni Q'_4 \\
@V{\frac{4}{2}}V{wt=w_1  }V @V{\frac{1}{2}}V{wt=w'_1}V\\
X @>=>> X
\end{CD}$$

Where $$ \begin{array}{ll} w_1=v_1=(4l+1,4l,2,1,8l+1), & w'_1=v_2=\frac{1}{2}(2l+1,2l+1,1,2,4l), \\
  w_2=\frac{1}{2(8l+1)}(6l+1,10l+1,1,12l+2,4l),  & w'_2=\frac{1}{2}(6l+1,6l-1,3,2,12l+2). \end{array}$$

It is easy to verify the condition $\sharp$ that
$$v_2= \frac{4l}{2(8l+1)}v_1+\widehat{w_2}, \quad v_1=\widehat{
w'_2}+v_2.$$

\noindent{\bf Case 2.} $v_1=(4l,4l-1,2,1,8l-1)$.\\
We first look at $Q_3$,
which is a quotient singularity of type $\frac{1}{4}(2,3,1,3,1)$ in
$\cX_1$.

\noindent
{\bf Claim.} $Q_3 \not \in Y$ and $x_3^{4l} \in \varphi_1$. \\
To see this, according to Kawakita's description, there is only one
non-hidden non-Gorenstein singularity and also the hidden
singularities has index at most $2$. Hence $Q_3 \not \in Y$. In
other words, one must have either $x_3^{4l} \in \varphi_1$ or
$x_3^{4l-1} \in \varphi_2$. Note that $x_3^{4l-1} \not \in \varphi_2$
otherwise $\varphi_2$ is not a semi-invariant. We thus conclude that
$x_3^{4l} \in \varphi_1$.

Next notice that $Q_5 \in \cX_1$ is a quotient singularity of type
$\frac{1}{2(8l-1)}(10l-1,6l-1,1,4l,12l-2)$. We set
$w_2=\frac{1}{2(8l-1)}(10l-1,6l-1,1,4l,12l-2)$ so that $v_2
=\frac{1}{2}(6l+1,6l-1,3,2,12l-2)$.

As before, the weighted blowup $\cX_2 \to \cX_1$ with vector $v_2$
gives a divisorial contraction $g: Z \to Y$ of discrepancy
$\frac{1}{2(8l-1)}$, which is compatible with the Kawamata blowup.
This can be seen by examining the local equation at $Q_5$ and the
weights as in Case 1.
$$ \left\{ \begin{array}{l} \overline{x_1}^2+\overline{x_4}+\overline{p}(\overline{x_2},\overline{x_3},\overline{x_4})=0, \\
\overline{x_2}^2+\overline{q}(\overline{x_1},\overline{x_3},\overline{x_4})+\overline{x_5}=0
\end{array} \right.$$

We now checked the numerical conditions for 2-ray game. We have
$$E^3= \frac{2}{2(8l-1)}, \quad F^3=\frac{(2(8l-1))^2}{(6l-1)(10l-1)},$$
and $$T(f,g)=\frac{1}{2(8l-1)} (-8+\frac{2}{10l-1} ) <0.$$

We pick $D_{0,X}=(x_3=0)$, which is  an elephant in $|-K_X|$. One
sees  that $E \cap D_0$ is defined by $\bZ_2$-quotient of the
complete intersection
$$\left\{
\begin{array}{l}
x_3=0,\\
{\varphi_{1,8l}}_{|x_3=0}=x_1^2+x_4x_5,\\
{\varphi_{2,8l-2}}_{|x_3=0}=x_2^2+ q_{8l-2}(x_1,0,x_4).$$
\end{array}
\right.$$

Same argument as in Case 1 shows that $D_0 \cap E$ is irreducible.
 Therefore, we can simply check
$$T(f,g)=\frac{1}{2(8l-1)} (-8+\frac{2}{10l-1} ) <0.$$
to conclude that $-K_Z/X$ is nef. This verifies Hypothesis
$\flat$(3).

Hence
$-K_Z/X$ is nef.

 The weighted blowup $\cX'_1 \to \cX_0$ with vector
$v_2$ gives a divisorial contraction $f': Y' \to X$ of discrepancy
$\frac{3}{2}$.
This can be seen to be a compatible re-embedding of Kawakita's description by eliminating $x_5$.

One sees that $v_1=\frac{2l-1}{2}
e_1+\frac{2l-1}{2}e_2+\frac{1}{2}e_3+v_2+\frac{4l}{2}e_5$.
Therefore, one consider the weighted blowup $\cX'_2 \to \cX'_1$ with
weights $w'_2=\frac{1}{2}(2l-1,2l-1,1,2,4l)$ over $Q'_4 \in \cX'_1$.
Let $Z'$ be the proper transform in $\cX'_2$, then one can easily
check that $Z'
\to Y'$ is a divisorial contraction over $Q'_4$ with discrepancy
$\frac{1}{2}$.

We summarize this case into following diagram.
$$\begin{CD}
Z @>{\dashrightarrow}>> Z' \\
@V{\frac{1}{2(8l-1)}}V{wt=w_2}V @V{\frac{1}{2}}V{wt=w'_2}V \\
Q_5 \in Y @.  Y' \ni Q'_4 \\
@V{\frac{4}{2}}V{wt=w_1  }V @V{\frac{3}{2}}V{wt=w'_1}V\\
X @>=>> X
\end{CD}$$

Where $$ \begin{array}{l} w_1=v_1=(4l,4l-1,2,1,8l-1),\\
  w_2=\frac{1}{2(8l-1)}(10l-1,6l-1,1,4l,12l-2),\\
  w'_1=v_2=\frac{1}{2}(6l+1,6l-1,3,2,12l-2), \\
   w'_2=\frac{1}{2}(2l-1,2l-1,1,2,4l). \end{array}$$

\subsection{discrepancy=$a/2$ over a cD/2 point}
Let $Y \to X$ be a divisorial contraction to a cD/2 point $P \in X$
with discrepancy $\frac{a}{2}$. This was classified by Kawakiata
into two cases (cf. \cite[Theorem 1.2.ii]{Kk05}).

\noindent{\bf Case 1.} In the case (a), the local equation is given
by
$$\varphi:x_1^2+x_2^2 x_4+x_1x_3q(x_3^2,x_4)+\lambda x_2x_3^{2\alpha-1}+p(x_3^2,x_4)=0 \subset \bC^4/v$$ with $v=\frac{1}{2}(1,1,1,0)$  and $f$ is the weighted blowup with
weights $v_1=\frac{1}{2}(r+2,r,a,2)$, where $r+1=2ad$ and both $a,r$
are odd. Notice that $wt_{v_1}(\varphi)=r+1$ and as observed in
\cite{CH}, we have that $x_3^{4d} \in p(x_3^2,x_4)$.

There are two quotient singularities $Q_1,Q_2$ of index $r+2,r$
respectively.

\noindent{\bf Subcase 1.} We first take $g: Z \to Y$  the  Kawamata
blowup at $Q_1$, which is of type $\frac{1}{r+2}(4d,1,r+2-4d)$. We
set $w_2=\frac{1}{r+2}(4d,4d,1,r+2-4d)$ so that the weighted blowup
$\cX_2 \to \cX_1$ with weights $w_2$ is compatible with $g$.

 One has
$$E^3=\frac{4(r+1)}{ar(r+2)}, \quad F^3= \frac{(r+2)^2}{4d(r+2-4d)}.$$

In this case, the naive choose of $D_{0,X}=(x_3=0) \in |-K_X|$ is
reducible. We therefore pick $D_{0,X}=(x_4=0)$ instead. It is
elementary to check that ${\varphi_{2l+2}}_{|x_4=0}=x_3^{4d}$. Hence
$E \cap D_0$ is  irreducible. We have $c_0=2, q_0=r+2-4d$, hence
$c_0-4 q_0 <0$ and

$$T(f,g,D_0)=\frac{1}{r+2}(-\frac{2(r+1)}{r}+1) <0.$$
Therefore $-K_{Z/X}$ is nef and Hypothesis $\flat$ holds.

We summarize this case into following diagram.
$$\begin{CD}
Z @>{\dashrightarrow}>> Z' \\
@V{\frac{1}{r+2}}V{wt=w_2}V @V{\frac{a-2}{2}}V{wt=w'_2}V \\
Q_1 \in Y @.  Y' \ni Q'_4 \\
@V{\frac{a}{2}}V{wt=w_1  }V @V{\frac{2}{2}}V{wt=w'_1}V\\
X @>=>> X
\end{CD}$$

Where $$ \begin{array}{ll} w_1=v_1=\frac{1}{2}(r+2,r,a,2), & w'_1=v_2=(2d,2d,1,1) \\
  w_2=\frac{1}{r+2}(4d,4d,1,r+2-4d),  & w'_2=\frac{1}{2}(r+2-4d,r-4d,a-2,2). \end{array}$$

Notice also that $f'$ is a divisorial contraction of the same type
over a cD/2 point with smaller discrepancy $\frac{a-2}{2}$, where
$r+1-4d=2d(a-2)$. The map $g'$ is a contraction with discrepancy $1$
which is in Case 1 of Subsection 3.4.

\noindent{\bf Subcase 2.} If we take $g: Z \to Y$ to be the awamata
blowup at $Q_2$, which is a quotient singularity of type
$\frac{1}{r}(4d,r-4d,1)$. We set $w_2=\frac{1}{r}(4d,r-4d,1,4d)$ so
that the weighted blowup $\cX_2 \to \cX_1$ with weights $w_2$ is
compatible with $g$.

One has
$$E^3=\frac{4(r+1)}{ar(r+2)}, \quad F^3= \frac{r^2}{4d(r-4d)}.$$

We pick $D_{0,X}=(x_4=0)$ as in Subcase 1, then we have
$c_0=2,q_0=4d$ and
$$T(f,g)=\frac{1}{r}(-\frac{2(r+1)}{r+2}+1) <0.$$
Therefore  $-K_{Z/X}$ is  nef and hence Hypotheis $\flat$ hold.


We summarize this case into following diagram.
$$\begin{CD}
Z @>{\dashrightarrow}>> Z' \\
@V{\frac{1}{r}}V{wt=w_2}V @V{\frac{2}{2}}V{wt=w'_2}V \\
Q_2 \in Y @.  Y' \ni Q'_4 \\
@V{\frac{a}{2}}V{wt=w_1  }V @V{\frac{a-2}{2}}V{wt=w'_1}V\\
X @>=>> X
\end{CD}$$

Where $$ \begin{array}{ll} w_1=v_1=\frac{1}{2}(r+2,r,a,2), & w'_1=v_2=\frac{1}{2}(r+2-4d,r-4d,a-2,2) \\
  w_2=\frac{1}{r}(4d,r-4d,1,4d),  & w'_2=(2d,2d,1,1). \end{array}$$

Notice also that $g'$ is a divisorial contraction of the same type
over a cD/2 point with smaller discrepancy $\frac{a-2}{2}$, where
$r+1-4d=2d(a-2)$. The map $f'$ is a contraction with discrepancy $1$
which is in Case 1 of Subsection 3.4.

\noindent {\bf Case 2.} In the case (b), the local equation is given
by
$$\left\{ \begin{array}{l}
\varphi_1=x_4^2+x_2x_5+p(x_1,x_3)=0 \\
\varphi_2=x_2 x_3 +x_1^{2d+1} + q(x_1,x_3) x_1x_3 + x_5 =0.
\end{array} \right\} \subset \bC^5/v
,$$ with $v=\frac{1}{2}(1,1,0,1,1)$ and $f$  is a weighted blowup
with weights $v_1=\frac{1}{2}(a,r,2,r+2,r+4)$ with $r+2=(2d+1)a$.
Notice that $a$ is allowed to be even in this case.

There are quotient singularities $Q_2,Q_5$ of index $r,r+4$
respectively.

\noindent{\bf Subcase 1.} We first consider the extraction $Z \to Y$
over $Q_5$, which is a quotient singularity of type
$\frac{1}{r+4}(1,r-2d+3,2d+1)$. We set
$w_2:=\frac{1}{r+4}(1,4d+2,r-2d+3,2d+1,2d+1)$, then its give rise to
a weighted blowup compatible with Kawamata blowup $g: Z \to Y$.

We check that
$$E^3 = \frac{4(r+2)}{ar(r+4)}, \quad F^3=\frac{(r+4)^2}{(2d+1)(r-2d+3)}.$$
We pick $D_{0,X}=(x_3=0)$ in this case. Then it is elementary to
check that $D_0 \cap E$ is irreducible. We have $c_0=2, q_0=r-2d+3,
q_i:=q_5=2d+1$ and
$$T(f,g,D_0)=\frac{1}{r+4}(-\frac{2(r+2)}{r}+1)<0.$$

We summarize this case into following diagram.
$$\begin{CD}
Z @>{\dashrightarrow}>> Z' \\
@V{\frac{1}{r+4}}V{wt=w_2}V @V{\frac{a-1}{2}}V{wt=w'_2}V \\
Q_5 \in Y @.  Y' \ni Q'_3 \\
@V{\frac{a}{2}}V{wt=w_1  }V @V{\frac{1}{2}}V{wt=w'_1}V\\
X @>=>> X
\end{CD}$$

Where $$ \begin{array}{l} w_1=v_1=\frac{1}{2}(a,r,2,r+2,r+4),\\
w_2=\frac{1}{r+4}(1,4d+2,r-2d+3,2d+1,2d+1), \\
w'_1=v_2=\frac{1}{2}(1,2d+1,2,2d+1,2d+1), \\
 w'_2=\frac{1}{2}(a-1,r-2d-1,2,r-2d+1,r-2d+3). \end{array}$$

Notice also that $g'$ is a divisorial contraction of the same type
over a cD/2 point with smaller discrepancy $\frac{a-1}{2}$, where
$r-2d+1=(2d+1)(a-1)$. The map $f'$ is a contraction with discrepancy
$\frac{1}{2}$ which is a compatible weighted blowup of
\cite[Proposition 5.8]{HaII} by eliminating $x_5$.




\noindent{\bf Subcase 2.} One can also consider $Z \to Y$ be the
Kawamata blowup over $Q_2$, which is a quotient singularity of type
$\frac{1}{r}(1,r-2d-1,2d+1)$. We set
$w_2:=\frac{1}{r}(1,r-2d-1,2d+1,2d+1,4d+2)$ so that the weighted
blowup is compatible with the Kawamata blowup $g$.

We check that
$$E^3 = \frac{4(r+2)}{ar(r+4)}, \quad F^3=\frac{r^2}{(2d+1)(r-2d-1)}.$$
We still pick $D_{0,X}=(x_3=0)$ in this case which  is known to be
irreducible. We have $c_0=2, q_0=2d+1, q_i:=r-2d-1$ and hence
$$T(f,g)=\frac{1}{r}(-\frac{2(r+2)}{r+4}+1)<0.$$
Therefore, Hypothsis $\flat$ holds.

We summarize this case into following diagram.
$$\begin{CD}
Z @>{\dashrightarrow}>> Z' \\
@V{\frac{1}{r}}V{wt=w_2}V @V{\frac{1}{2}}V{wt=w'_2}V \\
Q_2 \in Y @.  Y' \ni Q'_3 \\
@V{\frac{a}{2}}V{wt=w_1  }V @V{\frac{a-1}{2}}V{wt=w'_1}V\\
X @>=>> X
\end{CD}$$

Where $$ \begin{array}{l} w_1=v_1=\frac{1}{2}(a,r,2,r+2,r+4),\\
w_2=\frac{1}{r}(1,r-2d-1,2d+1,2d+1,4d+2), \\
w'_1=v_2=\frac{1}{2}(a-1,r-2d-1,2,r-2d+1,r-2d+3), \\
 w'_2=\frac{1}{2}(1,2d+1,2,2d+1,2d+1). \end{array}$$


\subsection{discrepancy 2/2 to a cE/2 point}
In this case, by \cite[Theorem 1.2]{Ha1}, the local equation is
$$\varphi:
x_4^2+x_1^3+x_2^4+x_3^8+... =0 \subset \bC^4/v,$$ with
$v=\frac{1}{2}(0,1,1,1)$.

By Hayakawa's result \cite{Ha1}, we know that $Y \to X$ is given by
weighted blowup with vector $v_1=(3,2,1,4)$. There is a quotient
singularity $Q_1$ of index $6$.

\noindent{\bf Remark.} There is another quotient singularity $R_3$
of index $2$ in the fixed locus of $\bZ_2$ action on $U_3$, which is
not $Q_3$.

We can take $w_2=\frac{1}{6}(2,5,1,1)$, then
$v_2=\frac{1}{2}(2,3,1,3)$. We pick $D_{0,X}=(x_3=0) \in |-K_X|$ and
it is easy to see that $D_0 \cap E$ is $\bZ_2$ quotient of
$(x_4^2+x_2^4=0) \subset \bP(3,2,1,4)$, which is irreducible.  We
also checked that
$$E^3=\frac{1}{6},\quad F^3=\frac{36}{5},\quad T(f,g)=\frac{-1}{10}<0.$$

We summarize this case into following diagram.
$$\begin{CD}
Z @>{\dashrightarrow}>> Z' \\
@V{\frac{1}{6}}V{wt=w_2}V @V{\frac{1}{3}}V{wt=w'_2}V \\
Q_1 \in Y @.  Y' \ni Q'_2 \\
@V{\frac{2}{2}}V{wt=w_1  }V @V{\frac{1}{2}}V{wt=w'_1}V\\
X @>=>> X
\end{CD}$$

Where $$ \begin{array}{ll} w_1=v_1=(3,2,1,4), & w'_1=v_2=\frac{1}{2}(2,3,1,3) \\
  w_2=\frac{1}{6}(2,5,1,1),  & w'_2=\frac{1}{3}(5,4,1,6). \end{array}$$

 Notice that $f':Y' \to X$ is the weighted blowup
with vector $v_2$ with discrepancy  $\frac{1}{2}$ as in
\cite[Theorem 10.41]{Ha1}. The point $Q'_2 \in Y'$ is a cD/3 point
with local equation
$$\overline{x_4}^2+\overline{x_1}^3+\overline{x_2}^3+\overline{x_3}^8\overline{x_2}+... =0 \subset
\bC^4/v,$$ with $v=\frac{1}{3}(2,1,1,0)$.  Hence $Z' \to Y'$ is the
weighted blowup with weights $w'_2$ with discrepancy $\frac{1}{3}$
as in \cite[Theorem 9.25]{Ha1}.

\subsection{discrepancy 2/2 to a cD/2 point}
There are three cases to consider according Hayakawa's
classification \cite[Theorem 1.1]{Ha1}. Note that the case of
Theorem 1.1.(iii) was treated in Subsection 3.2 already.

\noindent{\bf Case 1.} The case of Theorem 1.1.(i) in \cite{Ha1}.\\
In this case, the local equation is $$ x_1^2+x_2^2 x_4
+s(x_3,x_4)x_2x_3x_4+r(x_3)x_2+p(x_3,x_4)=0 \subset \bC^4/v,$$ with
$v=\frac{1}{2}(1,1,1,0)$. The map $f: Y \to X$ is given by weighted
blowup with vector $v_1=(2l,2l,1,1)$. Moreover,
$wt_{v_1}(\varphi)=2l$ and  $x_3^{4l} \in p(x_3,x_4)$.

There is a  singularity $Q_2$ of type $cA/4l$ with $aw=2$. The local
equation  in $U_2$ is given by $$ \overline{x_1}^2+\overline{x_2}
\overline{x_4} +\overline{x_3}^{4l}+... =0 \subset
\bC^4/\frac{1}{4l}(0,2l-1,1,2l+1).
$$
 Since $\overline{x_3}^{4l}$ appears in the
equation, in terms of the terminology as in \cite[\S 6]{HaI}, one
has $\tau-wt(\overline{x_3}^{4l})=1$. This implies that there is
only one weighted blowup $Z \to Y$ with minimal discrepancy
$\frac{1}{4l}$ which is given by
 the weight $w_2=\frac{1}{4l}(4l,2l-1,1,2l+1)$.

We pick $D_{0,X}=(x_3=0) \in |-K_X|$ and it is easy to see that $D_0
\cap E$ is $\bZ_2$ quotient of $(x_1^2+a_{0,4l}x_4^{4l}=0) \subset
\bP(3,2,1,4)$, where $a_{0,4l}$ denotes the coefficient. In any
event, this is irreducible.

We  checked that
$$E^3=\frac{2}{4l}, \quad F^3=\frac{(4l)^2}{(2l+1)(2l-1)}, \quad T(f,g)=\frac{1}{4l}(-2+\frac{1}{2l+1})<0.$$
Hence Hypothesis $\flat$ holds.

Hence we can summarize this case into following diagram.
$$\begin{CD}
Z @>{\dashrightarrow}>> Z' \\
@V{\frac{1}{4l}}V{wt=w_2}V @V{\frac{1}{2}}V{wt=w'_2}V \\
Q_2 \in Y @.  Y' \ni Q'_4 \\
@V{\frac{2}{2}}V{wt=w_1  }V @V{\frac{1}{2}}V{wt=w'_1}V\\
X @>=>> X
\end{CD}$$

Where $$ \begin{array}{ll} w_1=v_1=(2l,2l,1,1), & w'_1=v_2=\frac{1}{2}(2l+1,2l-1,1,2) \\
  w_2=\frac{1}{4l}(4l,2l-1,1,2l+1),  & w'_2=\frac{1}{2}(2l-1,2l+1,1,2). \end{array}$$

In this case, both $f'$ and $g'$ are divisorial contractions to a
cD/2 point as in \cite[Proposition 5.8]{HaII}.

\noindent{\bf Case 2.} The case of Theorem 1.1.(i') in \cite{Ha1}.\\
In this case, the local equation is $$\varphi: x_1^2+x_2 x_3
x_4+x_2^{4} +x_3^{2b}+x_4^c =0 \subset \bC^4/v,$$ with $b \ge 2, c
\ge 4$ and $v=\frac{1}{2}(1,1,1,0)$. The map $f: Y \to X$ is given
by weighted blowup with vector $v_1=(2,2,1,1)$. Moreover,
$wt_{v_1}(\varphi)=4$.

There is a  singularity $Q_2$ of type $cA/4$ with local equation  in
$U_2$ is given by $$ \overline{x_1}^2+\overline{x_3} \overline{x_4}
+\overline{x_2}^{4}+\overline{x_3}^{2b}\overline{x_2}^{2b-4}+\overline{x_4}^c
\overline{x_2}^{c-4} =0 \subset \bC^4/\frac{1}{4}(0,1,1,3).
$$

Since $\overline{x_2}^{4}$ appears in the equation,  one has
$\tau-wt=1$. This implies that there is only one weighted blowup $Z
\to Y$ with minimal discrepancy $\frac{1}{4}$ which is given by
 the weight $w_2=\frac{1}{4}(4,1,1,3)$.

 We pick $D_{0,X}=(x_3=0) \in |-K_X|$ again and it is easy to see that $D_0
\cap E$ is $\bZ_2$ quotient of $(x_1^2+\delta_{4,c}x_4^{c}=0)
\subset \bP(3,2,1,4)$, where $\delta_{4,c}$ is the Kronecker's delta
symbol. In any event, this is irreducible.

Then the invariant and diagram is exactly the same as the $l=1$ in
Case 1.  For reference, we have
$$E^3=\frac{2}{4}, \quad F^3=\frac{4^2}{3l}, \quad T(f,g)=\frac{1}{4}(-2+\frac{1}{3})<0.$$

 We summarize
the result into following diagram.
$$\begin{CD}
Z @>{\dashrightarrow}>> Z' \\
@V{\frac{1}{4}}V{wt=w_2}V @V{\frac{1}{2}}V{wt=w'_2}V \\
Q_2 \in Y @.  Y' \ni Q'_4 \\
@V{\frac{2}{2}}V{wt=w_1  }V @V{\frac{1}{2}}V{wt=w'_1}V\\
X @>=>> X
\end{CD}$$

Where $$ \begin{array}{ll} w_1=(2,2,1,1), & w'_1=\frac{1}{2}(3,1,1,2) \\
  w_2=\frac{1}{4}(4,1,1,3),  & w'_2=\frac{1}{2}(1,3,1,2). \end{array}$$

Note that $f', g'$ are the weighted blowup of type $v_1$ as in
\cite[\S 4]{HaII}.

\noindent{\bf Case 3.} The case of Theorem 1.1.(ii) in \cite{Ha1}.\\
The equation is given as
$$ \left\{ \begin{array}{l}
\varphi_1: x_1^2+x_4x_5+r(x_3)x_2+p(x_3,x_4)=0 \\
\varphi_2: x_2^2+s(x_3,x_4)x_1x_3+q(x_3,x_4)-x_5=0 \end{array}
\right. $$ with $v=\frac{1}{2}(1,1,1,0,0)$. The map $f:Y \to X$ is
given by weighted blowup with vector $v_1=(l+1,l,1,1,2l+1)$. We take
$Z \to Y$ to be the extraction over the quotient singularity $Q_5$,
which is a quotient singularity of type $\frac{1}{4l+2}(3l+2,l,1)$.

We can write $p(x_3,x_4)=p_0(x_3)+x_4p_1(x_3,x_4)$. By replacing
$x_5$ with $x_5-p_1(x_3,x_4)$, we may and so assume that $p=p(x_3)$.

We need to distinguish into two subcases according to the parity of
$l$.

\noindent{\bf Subcase 3.1} $l$ is odd. \\
In this situation, we need to use the fact that either $x_3^{2l+2}
\in \varphi_1$ or $x_2 x_3^{l+2} \in \varphi_1$ (cf. \cite[Theorem
1.1.ii.b,c]{Ha1}). By this fact, one sees that the compatible
weighted blowup is given by $w_2=\frac{2l}{4l+2}(3l+2,l,1,2l+2,2l)$.

We now pick $D_{0,X}=(x_3=0) \in |-K_X|$ again and it is easy to see
that $D_0 \cap E$ is $\bZ_2$ quotient of
$(x_1^2=x_2^2+a_{0,2l}x_4^{2l}=0) \subset \bP(l+1,l,1,1,2l+1)$,
where $a_{0,2l}$ is the coefficient. In any event, this is
irreducible.

We have
$$\begin{array}{l} E^3=\frac{4}{4l+2}, \quad F^3=\frac{(4l+2)^2}{l(3l+2)},\\
T(f,g)=\frac{1}{4l+2}(-4+\frac{2l}{l(3l+2)})<0. \end{array}$$

 We summarize the
result into following diagram.
$$\begin{CD}
Z @>{\dashrightarrow}>> Z' \\
@V{\frac{1}{4}}V{wt=w_2}V @V{\frac{1}{2}}V{wt=w'_2}V \\
Q_5 \in Y @.  Y' \ni Q'_4 \\
@V{\frac{2}{2}}V{wt=w_1  }V @V{\frac{1}{2}}V{wt=w'_1}V\\
X @>=>> X
\end{CD}$$

Where $$ \begin{array}{ll} w_1=(l+1,l,1,1,2l+1), & w'_1=\frac{1}{2}(l+2,l,1,2,2l) \\
  w_2=\frac{1}{4l+2}(3l+2,l,1,2l+2,2l),  & w'_2=\frac{1}{2}(l,l,1,2,2l-1). \end{array}$$

\noindent{\bf Subcase 3.2} $l$ is even. \\
In this situation, we need to use the fact that either $x_3^{2l} \in
\varphi_2$ or $x_1 x_3^{l-1} \in \varphi_2$ (cf. \cite[Theorem
1.1.ii.a]{Ha1}). Then the compatible weighted blowup is given by
$w_2=\frac{1}{4l+2}(l+1,3l+1,1,2l+2,2l)$.

We pick $D_{0,X}=(x_3=0) \in |-K_X|$ again such that $D_0 \cap E$ is
irreducible similarly. We have
$$\begin{array}{l} E^3=\frac{4}{4l+2}, \quad
F^3=\frac{(4l+2)^2}{(l+1)(3l+1)},\\
T(f,g)=\frac{1}{4l+2}(-4+\frac{2l}{(l+1)(3l+1)})<0. \end{array}$$

 We summarize the
result into following diagram.
$$\begin{CD}
Z @>{\dashrightarrow}>> Z' \\
@V{\frac{1}{4l+2}}V{wt=w_2}V @V{\frac{1}{2}}V{wt=w'_2}V \\
Q_5 \in Y @.  Y' \ni Q'_4 \\
@V{\frac{2}{2}}V{wt=w_1  }V @V{\frac{1}{2}}V{wt=w'_1}V\\
X @>=>> X
\end{CD}$$

Where $$ \begin{array}{ll} w_1=(l+1,l,1,1,2l+1), & w'_1=\frac{1}{2}(l+1,l+1,1,2,2l) \\
  w_2=\frac{1}{4l+2}(l+1,3l+1,1,2l+2,2l),  & w'_2=\frac{1}{2}(l+1,l-1,1,2,2l+2). \end{array}$$

\subsection{discrepancy  a/n to a cA/n point}
This case is described in \cite[Theorem 1.1.i]{Kk05}, the local
equation is given by
$$ \varphi: x_1x_2 +g(x_3^r,x_4)=0 \subset \bC^4/v,$$
where $v=\frac{1}{n}(1,-1,b,0)$.

The map $f$ is given by weighted blowup with weight
$v_1=\frac{1}{n}(r_1,r_2,a,r)$. We may write $r_1+r_2=dan$ for some
$d>0$ with the term $x_3^{dn} \in \varphi$. We also have that
$s_1:=\frac{a-br_1}{n}$ is relatively prime to  $r_1$ and
$s_2:=\frac{a+br_2}{n}$ is relatively prime to $r_2$ (cf.
\cite[Lemma6.6]{Kk05}). We thus have the following:

$$\left\{ \begin{array}{l}
a=br_1+ns_1,\\
1=q_1r_1+s_1^* s_1,\\
a=-br_2+ns_2,\\
1=q_2r_2+s_2^* s_2,\\
\end{array} \right.
$$
for some $0 \le s_i^* < r_i$ and some $q_i$.

We set
$$ \delta_1:=-nq_1+bs_1^*, \quad \delta_2:=-n q_2-bs_2^*.$$
One sees easily that
$$\left\{
\begin{array}{l}
\delta_1 r_1+n=as_1^*,\\
\delta_2 r_2+n=as_2^*.
\end{array}
\right.
$$

\noindent{\bf Claim 1.} $a> \delta_i \ne 0$ for $i=1,2$.\\
To see this, first notice that if $\delta_1=0$, then $s_1^*=tn,
q_1=tb$ for some integer $t$. It follows that $1=ta$, which
contradicts to $a>1$. Hence $\delta_1 \ne 0$ and similarly $\delta_2
\ne 0$.

Note that $\delta_ir_i=as_i^*-n < as_i^* <ar_i$. Hence we have
$\delta_i <a$ for $i=1,2$. This completes the proof of the Claim 1.


Moreover, we need the following: \\
\noindent {\bf Claim 2.} $\delta_i >0$ for some $i$.\\
If  $\delta_i <0$, then $n=-\delta_i r_i +as_i^* \ge r_i$. In fact,
the equality holds only when $s_i^*=0$, which implies in particular
that $r_i=1$. We can not have the equalities simultaneously for
$i=1,2$ otherwise, $r_1=r_2=1$ yields $2=r_1+r_2=adn \ge 2n \ge 4$.
Therefore
$$ 2n > r_1+r_2 =adn \ge 2n, $$
which is absurd. This completes the proof of the Claim.

\begin{rem}
Suppose that both $\delta_1, \delta_2>0$ and $(a,r_1)=1$, then we
have $\delta_1+\delta_2=a$. To see this, note that
$as_2^*=n+\delta_2 r_2=n+ \delta_2 (adn-r_1)$. Therefore,
$$ a( s_2^* - \delta_2 dn)= n +(-\delta_2) r_1.$$
By $(a,r_1)=1$ and comparing it with $a s_1^*=n+(\delta_1) r_1$, we
have $\delta_1=-\delta_2+ ta$ for some $t \in \bZ$. Since $ 0<
\delta_1+\delta_2 <2a$, it follows that $\delta_1+\delta_2=a$.
\end{rem}

\noindent{\bf Subcase 1.} Suppose that $\delta_1>0$.\\
Notice that $r_1=1$ implies that $s_1^*=1, q_1=1$ and hence
$\delta_1=-n$. Therefore, we must have $r_1>1$.
 Let $g: Z \to
Y$ be Kawamata blowup over $Q_1$, which is a quotient singularity of
type $\frac{1}{r_1}(r_1-s_1^*,1,s_1^*)$. We take
$w_2=\frac{1}{r_1}(r_1-s_1^*,dr,1,s_1^*)$ which is a compatible
weighted blowup.

We pick $D_{0,X}=(x_4=0)$ then $E \cap D_0$ is defined by
$x_1x_2+x_3^{dn}=0$ which is clearly irreducible.
 We have $c_0=n, q_0=s^*_1$ and hence
$c_0-a q_0=-\delta_1 r_1 <0$. Also
 $$\begin{array}{l} E^3=\frac{dr^2}{r_1r_2}, \quad
F^3=\frac{(r_1)^2}{s^*_1(r_1-s^*_1)},\\
T(f,g, D_0)=\frac{1}{r_1}(-\frac{adn}{r_2}+1)<0.
\end{array}$$
Hence Hypothesis $\flat$ holds

 We summarize the
result into following diagram.
$$\begin{CD}
Z @>{\dashrightarrow}>> Z' \\
@V{\frac{1}{r_1}}V{wt=w_2}V @V{\frac{\delta_1}{n}}V{wt=w'_2}V \\
Q_1 \in Y @.  Y' \ni Q'_4 \\
@V{\frac{a}{n}}V{wt=w_1  }V @V{\frac{a-\delta_1}{n}}V{wt=w'_1}V\\
X @>=>> X
\end{CD}$$

Where
$$ \begin{array}{ll} w_1=\frac{1}{n}(r_1,r_2,a,n), & w'_1=\frac{1}{n}(r_1-s_1^*,r_2-\delta_1dn+s_1^*,a-\delta_1,n) \\
  w_2=\frac{1}{r_1}(r_1-s_1^*,dn,1,s_1^*),  & w'_2=\frac{1}{n}(s_1^*, \delta_1 dn-s_1^*,\delta_1,n). \end{array}$$

Note that $0< a':=a-\delta_1 <a$ and both $f',g'$ are extremal
contractions with discrepancies $< \frac{a}{r}$.

\noindent{\bf Subcase 2.} Suppose that $\delta_2>0$.\\
Again, $r_2>1$ under this assumption. Let $g: Z \to Y$ be Kawamata
blowup over $Q_2$, which is a quotient singularity of type
$\frac{1}{r_2}(r_2-s_2^*,1,s_2^*)$. We take $w_2=\frac{1}{r_2}(dr,
r_2-s_2^*,1,s_2^*)$ which is a compatible weighted blowup.

We pick $D_{0,X}=(x_4=0)$ again  which is  irreducible.
 We have $c_0=n, q_0=s^*_2$ and hence
$c_0-a q_0=-\delta_2 r_2 <0$. Also
 $$\begin{array}{l} E^3=\frac{dr^2}{r_1r_2}, \quad
F^3=\frac{(r_2)^2}{s^*_2(r_2-s^*_2)},\\
T(f,g, D_0)=\frac{1}{r_2}(-\frac{adn}{r_1}+1)<0.
\end{array}$$
Hence Hypothesis $\flat$ holds

 We summarize the
result into following diagram.
$$\begin{CD}
Z @>{\dashrightarrow}>> Z' \\
@V{\frac{1}{r_1}}V{wt=w_2}V @V{\frac{\delta_2}{n}}V{wt=w'_2}V \\
Q_2 \in Y @.  Y' \ni Q'_4 \\
@V{\frac{a}{n}}V{wt=w_1  }V @V{\frac{a-\delta_2}{n}}V{wt=w'_1}V\\
X @>=>> X
\end{CD}$$

Where
$$ \begin{array}{ll} w_1=\frac{1}{n}(r_1,r_2,a,n), & w'_1=\frac{1}{n}(r_1+s_2^*-\delta_2dn,r_2-s_2^*,a-\delta_2,n) \\
  w_2=\frac{1}{r_2}(dn, r_2-s_2^*,1,s_2^*),  & w'_2=\frac{1}{n}(\delta_2 dn-s_2^*,s_2^*, \delta_2,n). \end{array}$$

It is easy to see that if $r_1 \ge r_2$, then $\delta_1>0$. Hence
extracting over $Q_1$ provides the desired factorization. Similar
argument holds if  $r_2 \ge r_1$. Therefore, one can  conclude that
Theorems  holds by extracting over the point of highest index.

\section{further remarks}
It is easy to see that our method also work for any divisorial
contraction to a point of index 1 which is a weighted blowup. Let us
take $f: Y \to X$ the weighted blowup with weight $(1,a,b)$ for
example, where $a<b$ are relatively prime. Write $ap=bq+1$. Then by
our method, one sees easily that $g : Z \to Y$ is weighted blowup
with weight $\frac{1}{b}(p,1,b-p)$ over $Q_3$. After 2-ray game, we
have that $g'$ is the weighted blowup with weight $(1,q,p)$ over
$Q'_1$ and $f'$ is the weighted blowup with weight $(1,a-q, b-p)$.
Also $Z \dashrightarrow Z'$ is a toric flip. All the other known
examples fit into our framework nicely as well.


We would like to raise the following

\begin{Prbm}
Can every $3$-fold divisorial contraction to a point  be realized as
a weighted blowup?
\end{Prbm}

Assuming the affirmative answer, then by the method we provided in
this article, it is reasonable to expect, as in Corollary 1.3, that
for any $3$-fold divisorial contraction $Y \to X$ to a singular
point $P \in X$ of index $r=1$ with discrepancy $a>1$,
there exists a sequence of birational maps $$Y=:X_n \dashrightarrow \ldots \dashrightarrow X_0=:X$$ such that each map $X_{i+1} \dashrightarrow X_{i}$ is one of the following:\\
\begin{enumerate}
\item a divisorial extraction over a singular point of index $r_i \ge 1$ with discrepancy $\frac{1}{r_i}$.

\item a divisorial contraction to a singular point of index $r_i \ge 1$ with discrepancy $\frac{1}{r_i}$.

\item a flip or flop.

\end{enumerate}

Together with the factorization result of \cite{CH}, we have the
following:
\begin{conj}
Let $Y \dashrightarrow X$ be a birational map which is flip, a
divisorial contraction to a point, or a divisorial contraction to a
curve.
 There exists a sequence of birational maps $$Y=:X_n \dashrightarrow \ldots \dashrightarrow X_0=:X$$ such that each map $X_{i+1} \dashrightarrow X_{i}$ is one of the following:\\

\begin{enumerate}
\item a divisorial extraction or contraction over a point with minimal discrepancy,

\item a blowup of  a lci curve.

\item a flop.

\end{enumerate}

\end{conj}


\end{document}